\documentclass{amsart}

\usepackage{amsmath,amssymb,mathrsfs}
\usepackage[arrow,matrix]{xy}

\newcommand\rI{\mathrm I}
\newcommand\rII{\mathrm {II}}
\newcommand\rIII{\mathrm {III}}
\newcommand\rIV{\mathrm {IV}}

\newcommand\sF{\mathscr F}

\newcommand\cO{\mathcal O}

\newcommand\cD{\mathcal D}
\newcommand\cC{\mathcal C}
\newcommand\cZ{\mathcal Z}
\newcommand\cW{\mathcal W}
\newcommand\cX{\mathcal X}
\newcommand\cY{\mathcal Y}
\newcommand\cA{\mathcal A}

\newcommand\cP{\mathcal P}
\newcommand\cQ{\mathcal Q}
\newcommand\cE{\mathcal E}

\newcommand\bA{\mathbb A}
\newcommand\bP{\mathbb P}
\newcommand\bZ{\mathbb Z}
\newcommand\bQ{\mathbb Q}
\newcommand\bC{\mathbb C}
\newcommand\bG{\mathbb G}
\newcommand\bF{\mathbb F}

\newcommand\fS{\mathfrak S}

\newcommand\numero{n\textsuperscript{o}\,}

\DeclareMathOperator{\Km}{Km}
\DeclareMathOperator{\Bl}{Bl}
\DeclareMathOperator{\Spec}{Spec}
\DeclareMathOperator{\charac}{char}
\DeclareMathOperator{\Pic}{Pic}
\DeclareMathOperator{\Cl}{Cl}
\DeclareMathOperator{\ord}{ord}
\DeclareMathOperator{\lcm}{lcm}
\DeclareMathOperator{\Gal}{Gal}

\DeclareMathOperator{\ch}{ch}
\DeclareMathOperator{\CH}{CH}
\DeclareMathOperator{\Tr}{Tr}

\DeclareMathOperator{\NS}{NS}

\newcommand\topo{\text{top}}
\newcommand\un{\text{un}}
\newcommand\sm{\text{sm}}
\newcommand\et{\mathrm{\acute et}}
\newcommand\crys{\mathrm{crys}}
\newcommand\st{\mathrm{st}}

\newcommand\id{\mathrm{id}}
\newcommand\cl{\mathrm{cl}}

\newcommand\POK{\bP_{\cO_K}}

\theoremstyle{plain}
\newtheorem{theorem}{Theorem}[section]
\newtheorem{lemma}[theorem]{Lemma}
\newtheorem{proposition}[theorem]{Proposition}
\newtheorem{corollary}[theorem]{Corollary}
\newtheorem{claim}[theorem]{Claim}
\theoremstyle{definition}
\newtheorem{remark}[theorem]{Remark}
\newtheorem{definition}[theorem]{Definition}

\begin{document}

\title[On good reduction of some K3 surfaces]
{On good reduction of some K3 surfaces related to abelian surfaces}
\author{Yuya Matsumoto}
\address{Graduate school of mathematical sciences, University of Tokyo}
\email{ymatsu@ms.u-tokyo.ac.jp}
\date{2012/02/11}

\begin{abstract}
The N\'eron--Ogg--\v Safarevi\v c criterion for abelian varieties tells that 
whether an abelian variety has good reduction or not 
can be determined from the Galois action on its $l$-adic \'etale cohomology.
We prove an analogue of this criterion 
for some special kind of K3 surfaces 
(those which admit Shioda--Inose structures of product type), 
which are deeply related to abelian surfaces.
We also prove a $p$-adic analogue.
This paper includes Ito's unpublished result for Kummer surfaces.
\end{abstract}

\maketitle

\section*{Introduction}
We consider the problem of determining
whether a variety over a local field have good reduction
in terms of the Galois action on the $l$-adic \'etale cohomology of the variety.

An ideal situation is 
the case of abelian variety:
the reduction type (good or bad) is completely determined 
by the Galois action on the (first) $l$-adic \'etale cohomology group 
(Theorem \ref{thm:critAV}). 
In 2001, Ito obtained an analogous result on Kummer surfaces (Theorem~\ref{thm:critetKum}).

In this paper, we prove analogous results for another class of K3 surfaces:
those which admit Shioda--Inose structures of product type
(see Definition~\ref{def:SIP}), 
which are closely related to abelian surfaces.

We state the main theorems. 
First let us fix the notation: 
Let $K$ be a local field (that is, a complete discrete valuation field with perfect residue field)
and denote by $\cO_K$ the ring of integers, by $p$ the residue characteristic, and by $G_K$ the absolute Galois group of $K$.
A proper smooth variety $X$ over $K$ is said to have \emph{good reduction} over $K$
if there exists a proper smooth scheme $\cX$ over $\cO_K$ having $X$ as the generic fiber.
A $G_K$-module is said to be \emph{unramified}
if the inertia subgroup $I_K$ of $G_K$ acts on it trivially.
Our first result is the following:

\begin{theorem}
 \label{thm:critetSI}
Let $K$ be a local field with residue characteristic $p \neq 2, 3$ 
and $l$ a prime number different from $p$.
Let $Y$ be a K3 surface over $K$ admitting a Shioda--Inose structure of product type.
If $H^2_\et(Y_{\overline K}, \bQ_l)$ is unramified, 
then $Y_{K'}$ has good reduction for some finite extension $K'$ 
which is purely inseparable\footnote{
 Of course there is no nontrivial purely inseparable extension if $\charac K = 0$.
} over a finite extension of $K$ of ramification index $1$, $2$, $3$, $4$ or $6$.
\end{theorem}

At present we do not know whether field extension is necessary.

We also prove results concerning $p$-adic cohomology,
for both Kummer surfaces and K3 surfaces with Shioda--Inose structure of product type.
Here, and whenever we mention $p$-adic cohomology, 
we assume that $K$ is of mixed characteristic $(0,p)$.

\begin{theorem}
 \label{thm:critcrysKum}
Let $K$ be a local field with residue characteristic $p \neq 2$ 
and $X$ a Kummer surface over $K$.
Assume that $X$ has at least one $K$-rational point.
If $H^2_\et(X_{\overline K}, \bQ_p)$ is crystalline, 
then $X_{K'}$ has good reduction for some finite unramified extension $K'/K$.
\end{theorem}

\begin{theorem}
 \label{thm:critcrysSI}
Let $K$ be a local field with residue characteristic $p \neq 2,3$ 
and $Y$ a K3 surface over $K$ with Shioda--Inose structure of product type.
If $H^2_\et(Y_{\overline K}, \bQ_p)$ is crystalline, 
then $Y_{K'}$ has good reduction for some finite extension $K'/K$
of ramification index $1$, $2$, $3$, $4$ or $6$.
\end{theorem}

As an immediate corollary of these theorems (and Theorem~\ref{thm:critgen}), we have a criterion for potential good reduction. 
The adjective ``potential'' means ``after taking a finite extension of the base field''.
\begin{corollary}
 \label{cor:potcritSI}
Let $X$ be a K3 surface which is one of the above two types. Then the following properties are equivalent: 
\begin{enumerate}
 \item The surface $X$ has potential good reduction.
 \item For some (any) prime $l \neq p$, 
  the second $l$-adic \'etale cohomology of $X$ is potentially unramified.
 \item The second $p$-adic \'etale cohomology of $X$ is potentially crystalline.
\end{enumerate}
\end{corollary}

There is an application to the reduction of singular K3 surface. 
Recall that a K3 surface over a field of characteristic $0$ is called \emph{singular} if it has the maximum possible Picard number $20$ 
(note that the word singular here does not mean non-smooth).

\begin{corollary}
 \label{cor:singk3}
Any singular K3 surface has potential good reduction.
\end{corollary}

\medskip

The structure of this paper is as follows.
In section~\ref{sec:prelim} 
we give some preliminary results.
We prove Theorem~\ref{thm:critetSI} in section~\ref{sec:lproof}
and Theorems \ref{thm:critcrysKum}, \ref{thm:critcrysSI} in section~\ref{sec:pproof}.
As an appendix, 
we give a proof of Ito's unpublished result (Theorem~\ref{thm:critetKum})
in section~\ref{sec:Kummer}.

\section{Preliminaries} \label{sec:prelim}

\subsection{General results} \label{subsec:prelim}

In this subsection,
we prove some basic results to be used later.

\begin{lemma}
Let $K$ be a local field, $\cO_K$ its ring of integers and $k$ its residue field.
Then there is an isomorphism $\Pic \POK^1 \cong \bZ$ 
compatible with $\Pic \POK^1 \rightarrow \Pic \bP^1_* \stackrel{\deg}{\rightarrow} \bZ$ 
for ${*} = K, k$.
\end{lemma}

\begin{proof}
$\Pic \POK^1$ is canonically isomorphic to 
the group $\Cl \POK^1$ of Weil divisors modulo principal divisors, 
and the same holds for $\bP^1_K$ and $\bP^1_k$.
A prime Weil divisor of $\POK^1$ is either the special fiber $\bP_k^1$ or of the form $V_+(P)$ 
where $P$ is a non-constant primitive homogeneous irreducible polynomial in $\cO_K[X,Y]$
(a polynomial in $\cO_K[X,Y]$ is primitive if its coefficients are not all divisible by the uniformizer of $\cO_K$).
The former is principal.
For the latter, $\sum n_i [V_+(P_i)]$ is principal if and only if $\sum n_i \deg P_i = 0$, 
and hence the degree map induces an isomorphism $\Cl \POK^1 \stackrel{\sim}{\rightarrow} \bZ$.
This is clearly compatible with $\Pic \POK^1 \rightarrow \Pic \bP^1_*$.
\end{proof}

\begin{lemma}
 \label{lem:pic-g}
Let $X$ be a geometrically connected variety over $F$
and assume that $X$ has at least one $F$-rational point.
Then
the natural map $\Pic X \rightarrow (\Pic X_{\overline F})^{G_F}$ is an isomorphism.
\end{lemma}

\begin{proof}
Recall that $\Pic X \cong H^1_\et(X, \bG_m)$ and $\Pic X_{\overline F} \cong H^1_\et(X_{\overline F}, \bG_m)$. 
We make use of the Hochschild--Serre spectral sequence 
\[
 E_2^{p.q} = H^p(G_F, H^q_\et(X_{\overline F}, \bG_m)) \Rightarrow H^{p+q}_\et(X_F, \bG_m).
\]
Since 
\[
 E^{1,0} = H^1(G_F, H^0_\et(X_{\overline F}, \bG_m)) = H^1_\et(\Spec F, \bG_m) = 0,
\]
we have an exact sequence
\[
 0 \rightarrow \Pic X \rightarrow (\Pic X_{\overline F})^{G_F} \rightarrow 
 H^2_\et(\Spec F, \bG_m) \rightarrow H^2_\et(X_F, \bG_m).
\]
By assumption that $X$ has at least one $F$-rational point, 
the morphism $X \rightarrow \Spec F$ has a section $s \colon \Spec F \rightarrow X$, 
which induces a splitting $s^* \colon H^2_\et(X, \bG_m) \rightarrow H^2_\et(\Spec F, \bG_m)$
of the last map in the above sequence. 
The map $H^2_\et(\Spec F, \bG_m) \rightarrow H^2_\et(X_F, \bG_m)$ is therefore injective. 
The conclusion follows from this.
\end{proof}

\begin{lemma}
 \label{lem:blupatsmpt}
Let $S$ be a scheme, $X$ a scheme over $S$ and $Z \subset X$ a closed subscheme of $X$. 
Assume that $X$ is smooth over $S$ and that the composite $Z \hookrightarrow X \rightarrow S$ is an isomorphism.
Then for any $S$-scheme $S'$, 
the canonical morphism $\Bl_{(Z \times_S S')} (X \times_S S') \rightarrow (\Bl_Z X) \times_S S'$ is an isomorphism.
\end{lemma}

\begin{proof}
An easy computation shows that the lemma is true if $X = \bA^d_S$ and $Z$ is the image of an $S$-valued point of $X$.
In the general case, 
since the assertion is local, 
we may assume that $X \rightarrow S$ factors $f \colon X \rightarrow X_0 = \bA^d_S$ with $f$ \'etale.
Let $Z_0$ be the scheme-theoretic image of $Z$ under $f$.
It follows that the composite $Z_0 \hookrightarrow X_0 \rightarrow S$ is an isomorphism 
and that $Z$ is an open and closed subscheme of $Y = X \times_{X_0} Z_0$.
Using the assertion for the case $X = \bA^d_S$ and the fact that blow-up commutes with flat base change, 
we obtain, for arbitrary $S' \rightarrow S$, 
\begin{align*}
   \Bl_{Y'} X' 
   \cong  (\Bl_{Z'_0} X'_0) \times_{X'_0} X'
  &\stackrel{\sim}{\rightarrow} (\Bl_{Z_0} X_0) \times_S S' \times_{X'_0} X' 
 \\
  &
   \cong (\Bl_{Z_0} X_0) \times_{X_0} X \times_S S'
   \cong (\Bl_Y X) \times_S S'
\end{align*}
(here the symbol $'$ means the base change by $S' \rightarrow S$).
The assertion follows from this
and the fact that $\Bl_Z X$ is isomorphic to $\Bl_Y X$ outside $Y \setminus Z$
and to $X$ outside $Z$
(and the corresponding fact for $\Bl_{Z'} X'$).
\end{proof}

\begin{lemma} 
 \label{lem:2-cover}
Let $F$ be a field of characteristic $\neq 2$ and $X$ be a connected smooth proper variety over $F$. 
Let $Z$ be an effective divisor on $X$ over $F$ with no multiple component. 
Then the class $[Z]$ of $Z$ in $\Pic(X)$ is divisible by $2$ if and only if 
there is a double covering $Y \rightarrow X$ whose branch locus is $Z$. 
If $\Pic(X)$ has no $2$-torsion
 and $F$ is algebraically closed, 
then such a covering is unique up to isomorphism.
\end{lemma}

\begin{proof}
Assume $D$ is a divisor such that $2[D] = [Z]$ in $\Pic(X)$.
Take representations $D = \{ (U_i, g_i) \}_i$ and $Z = \{ (U_i, f_i) \}_i$ 
with respect to a covering $X = \bigcup_i U_i$ by affine schemes.
By assumption there exist $c_i \in \cO_X(U_i)^*$ such that $g_i^2/g_j^2 = (f_i/f_j) \cdot (c_i/c_j)$.
Put $V_i = \Spec \cO_X(U_i) [T_i] / (T_i^2 - c_if_i)$. 
Then the natural morphism $V_i \rightarrow U_i$ is a double covering which branches at $Z \rvert_{U_i}$.
The morphisms $V_i \rightarrow U_i$ glue via $T_i = (g_i/g_j) T_j$
and we obtain a double covering $Y = \bigcup_i V_i \rightarrow X$ which branches at $Z$.

Conversely, 
assume $Y \rightarrow X$ is such a covering.
Then there exists a covering $X = \bigcup_i U_i$ by affine schemes 
and a representation $Z = \{ (U_i, f_i) \}_i$ 
such that locally $Y \rightarrow X$ is of the form
$\Spec \cO_X(U_i) [T_i] / (T_i^2 - c_i f_i) \rightarrow \Spec \cO_X(U_i)$
with $c_i \in \cO_X(U_i)^*$. 
Put $r_{ij} = T_i/T_j \in \cO_X(U_i \cap U_j)^*$. 
Then  any divisor $D = \{ (U_i, g_i) \}_i$ such that $g_i/g_j = r_{ij}$
 satisfies $2[D] = [Z]$.

For the second assertion, 
take two such coverings $Y \rightarrow X$ and $Y' \rightarrow X$. 
We define $c_i, c'_i \in \cO_X(U_i)^*$ and $r_{ij}, r'_{ij} = T'_i / T'_j \in \cO_X(U_i \cap U_j)^*$
as in the previous paragraph.
Since $\Pic(X) $ has no $2$-torsion, we have $r_{ij} / r'_{ij} = d_i / d_j$ for some $d_i \in \cO_X(U_i)^*$. 
Substituting this to the relations, 
we have $c_i / c'_i d_i^2 = c_j / c'_j d_j^2$ 
and these define an element of $\cO_X(X)^* = F^*$.
Since $F$ is algebraically closed (in particular closed under taking square roots), 
we can modify $d_i$ so that this element is $1$.
Then $Y$ and $Y'$ are isomorphic under the map $T_i \mapsto d_i T'_i$.
\end{proof}

\begin{lemma}
 \label{lem:tower}
{\rm (1)} 
 Let $K$ be a field (of positive characteristic).
 Let $L$ be a finite separable extension of degree $d$ of a finite purely inseparable extension of $K$.
 Then $L$ is finite purely inseparable over a finite separable extension of degree $d$ of $K$.

{\rm (2)} 
 The above statement remains true 
 if we replace ``field'' by ``local field'' and ``degree'' by ``ramification index''.
\end{lemma}

\begin{proof}
This is easy.
\end{proof}

\subsection{K3 surfaces}

Recall that K3 surface is a proper smooth minimal surface $X$ with $H^1(X, \cO_X) = 0$ and $\Omega^2_X \cong \cO_X$.

\begin{lemma}
 \label{lem:k3-pic}
Let $F$ be an algebraically closed field of characteristic $\neq 2$ 
and $X$ a K3 surface over $F$.
Then the following properties hold.
\begin{enumerate}
 \item \label{lem:k3-pic-fg}
  $\Pic(X)$ is finitely generated.
 \item \label{lem:k3-pic-2-tors}
  $\Pic(X)$ is $2$-torsion-free.
\end{enumerate}
\end{lemma}

\begin{proof}
(\ref{lem:k3-pic-fg})
The Picard group $\Pic(X)$ has a scheme structure over $F$ 
and the connected component $\Pic^0(X)$ of the identity 
is an abelian variety of dimension $\leq \dim H^1(X, \cO_X)$ (\cite[\numero 236, Proposition 2.10]{FGA}).
Since $X$ is a K3 surface we have $\dim H^1(X, \cO_X) = 0$.
Then the desired property follows from the fact by Kleiman~\cite[Expos\'e XIII, Th\'eor\`eme 5.1]{SGA6}
 that the N\'eron--Severi group $\Pic (X) / \Pic^0(X)$ is finitely generated.

(\ref{lem:k3-pic-2-tors}) 
For each $n \geq 1$, by the Kummer sequence 
\[
 0 \rightarrow \mu_{l^n} \rightarrow \bG_m \stackrel{l^n}{\rightarrow} \bG_m \rightarrow 0,
\]
we have an injection $\Pic(X) / l^n \Pic(X) \hookrightarrow H^2_\et(X, \mu_{l^n})$.
The inverse limit of these maps is also injective.
Since $\Pic(X)$ is finitely generated, $\varprojlim_n \Pic(X) / l^n \Pic(X) = \Pic(X) \otimes \bZ_l$.
So it suffices to show that $\varprojlim_n H^2_\et(X, \mu_{l^n}) = H^2_\et(X, \bZ_l)(1)$ is ($2$-)torsion free for $l=2$.

If $F$ is of positive characteristic, 
we can lift $X$ to characteristic 0 (\cite[Corollaire 1.8]{Deligne:relevement}). 
Since the singular cohomology $H^2(X, \bZ)$ of complex K3 surface is torsion-free (\cite[Proposition VIII.3.3]{BHPV:surfaces}), 
the above assertion follows from the proper base change theorem and the comparison theorem.
\end{proof}

An automorphism of a K3 surface is said to be \emph{symplectic} if it fixes a non-vanishing holomorphic $2$-form. 
(Note that, since the canonical divisor of a K3 surface is trivial, such a $2$-form exists and is unique up to constant multiple.)

The next lemma is important in studying symplectic involutions. 
This is a part of the result of Nikulin~\cite[section 5]{Nikulin:auto} for characteristic $0$, 
and Dolgachev--Keum~\cite[Theorem 3.3]{Dolgachev--Keum:auto} pointed out 
that Nikulin's argument stays valid for arbitrary characteristic $\neq 2$.
\begin{lemma}
 \label{lem:nikulin-inv}
 Let $\iota$ be a symplectic involution of a K3 surface $X$
 over an algebraically closed field of characteristic $\neq 2$.
 Then $\iota$ fixes exactly eight points 
 and $X / \langle \iota \rangle$ is birational to a K3 surface.
\end{lemma}

Next propositions are useful when we want K3 surfaces to have elliptic surface structures.

\begin{proposition}
 [{Pjatecki\u\i-\v Sapiro--\v Safarevi\v c~\cite[\S 3, Theorem 1]{PSS}}]
 \label{prop:PSS}
Let $F$ be a field of characteristic $\neq 2,3$.
Let $X$ be a K3 surface over $F$ and $D$ a nef effective divisor on $X$ 
satisfying $D^2 = 0$. 
Then the linear system $\lvert D_{\overline F} \rvert$ over $\overline F$ contains a divisor of the form $mC$ 
where $m>0$ and $C$ is an elliptic curve (over $\overline F$).
\end{proposition}

\begin{proposition}
 \label{prop:makeellfib}
{\rm (1)} 
Let $X$ and $D$ be as in the previous proposition and 
$Z$ a smooth rational curve on $X$.
Assume that 
$D$ is connected and 
$Z \cdot D = 1$.
Then 
 $\lvert D \rvert$ gives an elliptic fibration 
$X \rightarrow \bP^1$ 
having $Z$ as the image of a section.

{\rm (2)}
Assume moreover that $D'$ is another effective divisor on $X$ satisfying the same conditions as $D$
and that each component of $D'$ does not meet $D$.
Then $D'$ is another fiber of the elliptic fibration given in {\rm (1)}.
\end{proposition}

\begin{proof}
(1)
Let $m$ and $C$ be as in Proposition~\ref{prop:PSS}. 
Then $m$ divides $mC \cdot Z = D \cdot Z = 1$ and hence $m = \pm 1$. 
By effectiveness assumption we have $m = 1$.

By the same argument as in \cite[\S 3]{PSS}, 
we have $\dim_k \lvert D \rvert = 2$
and hence a morphism $\Phi \colon X \rightarrow \bP^1$.
Then $\Phi$ is an elliptic fibration 
since by the previous proposition $\Phi$ has a geometric fiber which is an elliptic curve.
By construction $D$ is a fiber.
Since $Z \cdot D = 1$, the composite $Z \hookrightarrow X \rightarrow \bP^1$ is an isomorphism, 
hence its inverse is a section.

(2)
Since each component of $D'$ does not intersect $D$, 
it is mapped to a point by $\Phi$.
Since $D'$ is connected, every component goes to the same point $p$.
Such a divisor has self-intersection $0$ 
if and only if it is a rational multiple of the whole fiber $\Phi^{-1}(p)$ 
(this follows from an elementary computation, 
or see \cite[Proposition III.8.2]{Silverman:AAEC}).
Comparing the intersection numbers with $Z$, it follows that $D'$ coincides with $\Phi^{-1}(p)$.
\end{proof}

In the following two lemmas, 
we consider K3 surfaces over a local field $K$. 
We 
denote by $l$ a prime different from the residue characteristic of $K$.

\begin{lemma}[{\cite[Lemmas 2.1 and 2.4]{Ito:Kummer}}]
  \label{lem:-2curves} 
Let $X$ be a K3 surface over $K$.
Assume that $H^2_\et(X_{\overline K}, \bQ_l)$ is unramified.
Then the following holds.
\begin{enumerate}
 \item \label{lem:-2curves:curve}
  Let $C \subset X_{\overline K}$ be a smooth rational curve. 
  Then $C$ is defined over 
  a finite extension which is purely inseparable over an unramified extension of $K$.
 \item \label{lem:-2curves:cover}
 Assume that $X$ has a $K$-rational point.
 Let $X' \rightarrow X_{\overline K}$ be a double covering
  ramified on $\bigcup_i C_i \subset X_{\overline K}$
  where each $C_i$ is a smooth rational curve.
  Then $X' \rightarrow X_{\overline K}$ is defined over 
  a finite extension which is purely inseparable over an unramified extension of $K$.
\end{enumerate}
\end{lemma}

\begin{proof}
(\ref{lem:-2curves:curve})
Recall that there exists the cycle map $\cl \colon Z^1(X_{\overline K}) \rightarrow H^2_\et(X_{\overline K}, \bQ_l)(1)$
which is compatible with the Galois action and the intersection pairing.
Take any $\sigma \in I_K = G_{K^\un}$.
By the unramifiedness assumption, $\sigma$ acts trivially on the image of $\cl$.
Therefore we have $C \cdot \sigma(C) = C \cdot C$.
By the adjunction formula, this value is equal to $-2$.
Since distinct curves cannot have negative intersection number, 
we have $\sigma(C) = C$.
Since this holds for any $\sigma \in G_{K^\un}$, 
it follows that $C$ is defined over $(K^\un)^{p^{-\infty}}$ and hence 
over an extension of desired type.

(\ref{lem:-2curves:cover})
The divisor $\bigcup C_i$ is defined over $(K^\un)^{p^{-\infty}}$ since each $C_i$ is defined over $(K^\un)^{p^{-\infty}}$ by (\ref{lem:-2curves:curve}).
By Lemma~\ref{lem:2-cover} we can take $Y \in \Pic X_{\overline K}$ such that $2Y = \bigl[ \bigcup C_i \bigr]$.
Then since $\Pic X_{\overline K}$ has no $2$-torsion (Lemma~\ref{lem:k3-pic}), $Y$ is $G_{K^\un}$-invariant. 
Since $X$ has a $K$-rational point,  $Y$ is in $\Pic X_{(K^\un)^{p^{-\infty}}}$ by Lemma~\ref{lem:pic-g}. 
This shows, by Lemma~\ref{lem:2-cover} again, 
that $X' \rightarrow X_{\overline K}$ is defined over $(K^\un)^{p^{-\infty}}$ and hence over an extension of desired type.
\end{proof}

\begin{remark}
 \label{rem:-2curves2}
By a similar argument, we have the following:
for a K3 surface $X$ over a field $F$ 
and a field $L$ containing $\overline F$, 
any smooth rational curve $C$ on $X_{\overline L}$ 
is defined over $\overline F$.
\end{remark}

\begin{lemma}
 \label{lem:prodabelsurf} 
Let $A$ be an abelian surface over $K$
such that $H^2_\et(A_{\overline K}, \bQ_l)$ is unramified.
If $A_{\overline K}$ is isomorphic to the product of two elliptic curves, 
then so is $A_{K'}$ for 
 some finite extension $K'$ which is purely inseparable over an unramified extension of $K$.
\end{lemma}

\begin{proof}
Take a decomposition $A_{\overline K} = C_1 \times C_2$.
By a similar argument as in the proof of Lemma~\ref{lem:-2curves}~(\ref{lem:-2curves:curve}), 
it follows that $C_1 \cdot \sigma(C_1) = C_1 \cdot C_1 = 0$ for any $\sigma \in I_K = G_{K^\un}$.
The origin of $A$ is in $C_1$, and since it is a $K$-rational point, it is also in $\sigma(C_1)$.
It follows that $\sigma(C_1) = C_1$ since otherwise $C_1 \cdot \sigma(C_1)$ should be $\geq 1$.
This means that $C_1$ is defined over $(K^\un)^{p^{-\infty}}$ and hence over a finite subextension $K'$ of $(K^\un)^{p^{-\infty}}$. 
Similarly for $C_2$.
Then the addition map $C_1 \times C_2 \rightarrow A_{K'}$ is an isomorphism. 
\end{proof}

\medskip

In this paper we consider two specific class of K3 surfaces: 
(1) Kummer surfaces and
(2) K3 surfaces which admit Shioda--Inose structures of product type.

\begin{definition}
A K3 surface $X$ over a field $F$ of characteristic $\neq 2$ is a \emph{Kummer surface} if, 
for some abelian surface $A$ over $\overline F$, 
$X_{\overline F}$ is isomorphic to the minimal desingularization $\Km A$ 
of the quotient surface $A/\langle -1 \rangle$ of $A$ by the multiplication-by-$(-1)$ map.
\end{definition}

\begin{definition} \label{def:SIP}
We say that 
a K3 surface $Y$ over a field $F$ admits a \emph{Shioda--Inose structure of product type}
if $Y_{\overline F}$ admits an elliptic fibration $\Phi \colon X_{\overline F} \rightarrow \bP^1_{\overline F}$ 
which admits a section and two (singular) fibers of type $\rII^*$ (in Kodaira's notation).
\end{definition}
\begin{remark}
The usual notion of Shioda--Inose structure is as follows:
a K3 surface $Y$ over $\bC$ admits a \emph{Shioda--Inose structure} if 
there exists a (necessarily symplectic) involution $\iota$ of $Y$ 
such that the minimal desingularization $X$ of the quotient surface $Y / \langle \iota \rangle$
is the Kummer surface of an abelian surface $A$
and that the quotient maps induce a Hodge isometry $T_Y \cong T_A$ ($T$ denotes the transcendental lattice of a surface).

A K3 surface $Y$ over $\bC$ 
admits a Shioda--Inose structure of product type in the sense of Definition~\ref{def:SIP}
if and only if 
it admits a Shioda--Inose structure in this sense with the corresponding abelian surface $A$ being the product of two elliptic curves
(for a proof of this assertion, 
see Shioda--Inose~\cite[Theorem 3]{Shioda--Inose}).
We prefer Definition~\ref{def:SIP} since it is valid for arbitrary base field.
\end{remark}

One may ask when or how often a K3 surface admits a Shioda--Inose structure, 
and when it is of product type. 

Na\"ively thinking, 
since the K3 surfaces which admit Shioda--Inose structures (resp.~those of product type) 
are in one-to-one correspondence 
to the abelian surfaces (resp.~product abelian surfaces), 
they form a $3$-dimensional (resp.\ $2$-dimensional) moduli.

Another answer (for surfaces over $\bC$) is the following criterion in terms of transcendental lattice:
a K3 surface $X$ over $\bC$ admits a Shioda--Inose structure if and only if
there exists a primitive embedding $T_X \hookrightarrow U^3$ (Morrison~\cite[Theorem 6.3]{Morrison}), 
and it is of product type if and only if 
there exists a primitive embedding $T_X \hookrightarrow U^2$ 
(\cite[Theorem 6.3]{Morrison} combined with \cite[Corollary 3.5]{Ma:decomposition}).
Here $U$ denotes the hyperbolic plane, the lattice of rank $2$ generated by $e_1$, $e_2$ 
with $e_i \cdot e_j = 1 - \delta_{ij}$.
In particular, if $X$ admits a Shioda--Inose structure (resp.~of product type) then
its Picard number is at least $17$ (resp.~at least 18).

\subsection{Known criteria for good reduction}

We recall the relation between cohomology and reduction of varieties over local fields, 
and the criteria for good reduction of abelian varieties.
In this subsection $K$ is a local field.

For general varieties, we have the following 
necessary condition for good reduction.
\begin{theorem}
 \label{thm:critgen}
Let $X$ be a variety over $K$
which has good reduction. 
Then the following properties hold.
\begin{enumerate}
 \item 
  {\rm (consequence of the smooth base change theorem \cite[Expos\'e XVI]{SGA4-3})}
  For any prime $l \neq p$, 
  the $l$-adic \'etale cohomologies of $X$ are unramified.
 \item 
  {\rm (consequence of the crystalline conjecture 
    (Faltings~\cite[Theorems 5.3 and 5.6]{Faltings:crys} and Tsuji~\cite[Theorem 0.2]{Tsuji:Cst}))}
  The $p$-adic \'etale cohomologies of $X$ are crystalline.
\end{enumerate}
\end{theorem}

For abelian varieties, this condition is also sufficient.

\begin{theorem}
 \label{thm:critAV}
Let $X$ be an abelian variety over $K$. 
Then $X$ having good reduction is equivalent to each of the following.
\begin{enumerate}
 \item 
  {\rm (N\'eron--Ogg--\v Safarevi\v c criterion, Serre--Tate~\cite[Theorem 1]{Serre--Tate})}
  For some (any) prime $l \neq p$, 
  the first (all) $l$-adic \'etale cohomology of $X$ is unramified.
 \item 
  {\rm (Coleman--Iovita~\cite[Theorem 4.7]{Coleman--Iovita})}
  The first (all) $p$-adic \'etale cohomology of $X$ is crystalline.
\end{enumerate}
\end{theorem}

The next result of Ito is an analogue of the above criterion for Kummer surfaces.
Since his paper is unpublished, 
we include (under his permission) the proof of this theorem in this paper as an appendix (section~\ref{sec:Kummer}).

\begin{theorem}[{Ito~\cite[Corollary 4.3]{Ito:Kummer}\footnote{
 In Ito's paper it was assumed that $\charac K = 0$.
}}] 
 \label{thm:critetKum}
Let $K$ be a local field with residue characteristic $p \neq 2$ 
and $l$ a prime number different from $p$.
Let $X$ be a Kummer surface over $K$. 
Assume that $X$ has at least one $K$-rational point.
If $H^2_\et(X_{\overline K}, \bQ_l)$ is unramified, 
then
$X_{K'}$ has good reduction for some finite extension $K'$
which is purely inseparable over an unramified extension of $K$.
\end{theorem}

\section{Proof of the $l$-adic result} \label{sec:lproof}

In this section we prove Theorem~\ref{thm:critetSI} and Corollary~\ref{cor:singk3}. 

For simplicity, throughout this section, we argue as if $\charac K = 0$ 
and omit the expression ``purely inseparable over'' coming from  Lemmas \ref{lem:-2curves} and \ref{lem:prodabelsurf}.
If $K$ is of positive characteristic this expression should be added;
by Lemma~\ref{lem:tower}, this does not affect the argument on the separable degree and the ramification index.

Since the statement of Theorem~\ref{thm:critetSI} 
admits 
finite unramified extensions, 
we often use the same symbol $K$ for finite unramified extensions of (the original) $K$.

We first outline the proof of Theorem~\ref{thm:critetSI} briefly. 
Let $Y$ be as in the statement of the theorem.
It is known that there exist rational maps 
$Y_{\overline K} \rightarrow X_{\overline K}$ and $X_{\overline K} \rightarrow Y_{\overline K}$ of degree $2$
for some Kummer surface $X_{\overline K}$ defined over $\overline K$.
We 
\begin{enumerate}
 \item  
  analyze the former map and
  construct a model $X$ of $X_{\overline K}$ over a finite unramified extension of $K$, 
 \item
  (using the unramifiedness of $H^2_\et(Y_{\overline K}, \bQ_l)$)
  show that $H^2_\et(X_{\overline K}, \bQ_l)$ is unramified
  after taking a finite extension of $K$ of ramification index $1$, $2$, $3$, $4$ or $6$, 
 \item
use Ito's result to obtain a good model (that is, a smooth proper model) $\cX$ of $X$ after taking a finite unramified extension of $K$, 
and 
 \item
analyze the latter map 
to construct a smooth $\cO_K$-scheme $\cY$, 
which will be a good model of $Y$.
\end{enumerate}
(The use of two different rational maps 
$Y_{\overline K} \rightarrow X_{\overline K}$ and $X_{\overline K} \rightarrow Y_{\overline K}$ seems to be essential.
See Remark~\ref{rem:2morph}.)

\medskip

Step (1). 
Let $Y$ be a K3 surface admitting a Shioda--Inose structure of product type. 
Then by definition
$Y_{\overline K}$ admits an elliptic fibration 
$\overline \Phi \colon Y_{\overline K} \rightarrow \bP^1_{\overline K}$
with a section 
and two singular fibers $D$, $D'$ of type $\rII^*$.
We will show that this fibration is defined over some finite unramified extension of $K$.

The singular fiber $D = \sum n_i C_i$ 
consists of $9$ smooth rational curves $C_1, \ldots, C_9 \subset X_{\overline K}$.
The image $Z$ of a section of $\Phi$ is also a smooth rational curve on $X_{\overline K}$. 
By Lemma~\ref{lem:-2curves}~(\ref{lem:-2curves:curve}), these curves are defined over a finite unramified extension of $K$ 
(which again we denote by $K$ for convenience).
Then by Proposition~\ref{prop:makeellfib}, 
there exists a unique elliptic fibration $\Phi \colon Y \rightarrow \bP^1$ defined over $K$ 
with a section $\bP^1 \stackrel\sim\rightarrow Z \hookrightarrow Y$ 
and singular fibers $D$ and $D'$.

\begin{claim}
Under this situation,  $Y$ is (generically) defined by an equation of the form
\[
 y^2 = x^3 + ax + (b_{-1}t^{-1} + b_0 + b_1t).
  \tag{\textasteriskcentered} 
\]
for some $a, b_{-1}, b_0, b_1 \in K$ with $b_{-1}, b_1 \neq 0$.
\end{claim}

\begin{proof}
Since $Y$ is an elliptic surface over $\bP^1$, 
it is generically defined by a minimal Weierstrass form \[ y^2 = x^3 + A(t)x + B(t) \] 
(in $\bP^2 \times \bP^1$ with coordinates $(x,y), t$) with $A, B \in K[t]$. 
Comparing the (topological) Euler numbers of $Y$ and the singular fibers 
(Kodaira~\cite[Theorem 12.2]{Kodaira:cptanalsurf}), 
we have\footnote
 {To be precise, Kodaira proved this for complex varieties and $\chi = \chi_\topo$. 
 The whole thing remains valid for arbitrary base field and $\chi = \chi_l$
 provided that the characteristic is different from $2,3$.}
\[ \max\{ 3 \deg A, 2 \deg B \} \leq \chi(Y) = 24 \]
where $\chi(Y)$ is the Euler number of $Y$. 
Hence we have $\deg A \leq 8$ and $\deg B \leq 12$. 
We may assume that the singular fibers of type $\rII^*$ are 
above $t = 0, \infty$.
Since the fiber above $t=0$ is of type $\rII^*$, 
we have $\ord_t A \geq 4$ and $\ord_t B = 5$.
Similarly, since the fiber above $t=\infty$ is of type $\rII^*$, 
we have $\deg A \leq 8-4 = 4$ and $\deg B = 12-5 = 7$.
Consequently we obtain a (generic) equation
\[
 y^2 = x^3 + at^4x + (b_{-1}t^{5} + b_0t^6 + b_1t^7)
\]
with $b_{-1}, b_1 \neq 0$.
The true defining equation (which we omit) is obtained by performing successive blow-ups on the above formula.
By a simple change of variables, we obtain the desired equation
\[
 y^2 = x^3 + ax + (b_{-1}t^{-1} + b_0 + b_1t).
\]
(This argument is similar to the one given by Shioda~\cite[section 4]{Shioda:sandwich}. 
However, since we are working on a field not algebraically closed, 
our formula is slightly more complicated than his.)
\end{proof}

Using the coordinates of (\textasteriskcentered) above,
we define an involution $\iota \colon Y \rightarrow Y$ by $(x,y,t) \mapsto (x, -y, b'/t)$
where $b' = b_{-1}/b_1$.
Then the fixed points of $\iota$ (over $\overline K$) are exactly the $2$-torsion points\footnote{
One might notice that $\Phi^{-1}(\pm \beta)$ may not be smooth (elliptic). 
However, some calculation shows that only singular fibers of types $\rI_2$ and $\rIV$ can occur. 
In these cases the number of $2$-torsion points is indeed four. 
} of $\Phi^{-1}(\pm \beta)$ (there are four for each)
where $\beta \in \overline K$ is a square root of $b'$.
The quotient $Y/\langle\iota\rangle$ has $8$ double points (over $\overline K$) and
its minimal desingularization $X$ is a K3 surface (Lemma~\ref{lem:nikulin-inv}); 
in fact, it is the Kummer surface
which appears in the definition of Shioda--Inose structure, 
and the corresponding abelian surface is 
(after taking base change to the algebraic closure) the product of two elliptic curves (Shioda~\cite[Theorem 1.1]{Shioda:sandwich}).

\medskip

Now we proceed to Step (2).
The \'etale cohomology of $X$ is given by
\[
 H^2_\et(X_{\overline K}, \bQ_l) \cong H^2_\et(Y_{\overline K}, \bQ_l)^{\langle \iota \rangle} 
  \oplus \bigoplus_{i=1}^8 \bQ_l(-1) [E_i].
\]
Here the last term is the Tate twist of the permutation representation 
corresponding to the eight $\iota$-fixed points (or the eight exceptional curves of $X \rightarrow Y / \langle \iota \rangle$).
Let $H \subset G_K$ be the kernel of this permutation action
and $K_H/K$ the corresponding (finite) extension.
Then the inertia subgroup of $G_{K_H}$ acts on $H^2_\et(X_{\overline K}, \bQ_l)$ trivially.
(In the $\charac K = 0$ case, this is the only place we need a (possibly) ramified extension.) 
In order to estimate the ramification index $f$ of $K_H/K$, we use the next lemma.
\begin{lemma}
Let $E$ be an elliptic curve defined over a local field $K$ of residue characteristic $\neq 2,3$. 
Then $K(E[2])/K$ is (at worst) tamely ramified and of ramification index at most $3$.

The same holds if $E$ is a singular fiber of type $\rI_2$ or $\rIV$.
\end{lemma}
\begin{proof}
We prove only the elliptic case, the remaining cases being similar.
Since $E[2] \setminus \{0\}$ consists of $3$ points, 
the extension $K(E[2])/K$ 
has a Galois group isomorphic to a subgroup of $\fS_3$, 
and in particular has order dividing $6$.
Hence 
the ramification is (at worst) tame and therefore 
the inertia group of $K(E[2])/K$ is cyclic.
A cyclic subgroup of $\fS_3$ is of order at most $3$.
\end{proof}

An element of $G_K$ belongs to $H$ if and only if 
it fixes $\beta$ and it fixes each $2$-torsion point of both $E_+$ and $E_-$,
where $E_\pm$ are the fibers of $\Phi$ above $\pm \beta \in \bP^1$.
Let $f_\pm$ be the ramification indices of $K'(E_\pm [2])/K'$
where $K' = K(\beta)$.
By the lemma we have $f_\pm \leq 3$.

If $K' = K$, 
then $K_H$ is the compositum of $K(E_\pm [2])$ and hence 
of ramification index equal to $\lcm(f_+, f_-)$ (by tameness).

If $K' \neq K$, 
then $E_+$ and $E_-$ are conjugate under the nontrivial element of $\Gal(K'/K)$ and hence 
$K'(E_+ [2])$ and $K'(E_- [2])$ have the same ramification index $f_+ = f_-$ over $K'$.
By tameness again, $K_H = K'(E_\pm[2])$ has ramification index $f_\pm$ over $K'$.
Hence $K_H$ has ramification index $f_\pm$ or $2f_\pm$ over $K$.

In each case we have $f \in \{ 1,2,3,4,6\}$.

\medskip

Step (3).
Since $H^2_\et(X_{\overline K}, \bQ_l)$ is unramified as a representation of $G_{K_H}$,
we can use Theorem~\ref{thm:critetKum} of Ito 
to obtain a good model $\cX$ over $\cO_{K'}$
for some finite unramified extension $K'$ of $K_H$.
(Note that $X$ has a $K$-rational point, since 
the intersection point of $Z$ and $D$ is a $K$-rational point on $Y$ 
and hence the corresponding exceptional curve on $X$ is isomorphic to $\bP^1_K$.)

Furthermore, by Lemma~\ref{lem:prodabelsurf}, 
the abelian surface $\cA$ over $\cO_{K'}$ appearing in the proof of Theorem~\ref{thm:critetKum} (see the proof of Lemma~\ref{alem:Kum})
is the product of two elliptic curves over $\cO_{K'}$ (after replacing $K'$ by a finite unramified extension).
This fact is used in the next step to obtain ``rational curves'' on $\cX$.

Hereafter we write simply $K$ instead of $K'$ 
(but note that this is a (possibly ramified) extension of the original $K$).

\medskip

Now we turn to Step (4): the construction of a good model $\cY$ from $\cX$.
This is the longest part of the proof.
We first recall the construction of Shioda \cite[Theorem 1.1]{Shioda:sandwich},
which describes $Y$ (up to desingularization) as a double quotient (instead of a double cover) of $X$, 
and then extend this construction to the relative case (that is, over $\cO_K$).

Fix a numbering 
$C_1[2] = \{p_i\}_{0 \leq i \leq 3}$ and 
$C_2[2] = \{q_j\}_{0 \leq j \leq 3}$: 
since $C_1$ and $C_2$ are defined and have good reduction over $K$, these points are defined over $K$ (after some finite unramified extension).
The surface $X = \Km(C_1 \times C_2)$ has $24$ specific rational curves (defined over $K$):
$u_i$, the strict transforms of the images of $p_i \times C_2$ under the quotient map; 
$v_j$, that of $C_1 \times q_j$; 
and the exceptional curves $w_{ij}$ 
corresponding to the images of $p_i \times q_j$.
The configuration of these curves are displayed in Fig.~\ref{fig:uvw}.
We focus on three divisors
\begin{align*}
 D_0      &= v_0 + v_1 + v_2 + 2w_{30} + 2w_{31} + 2w_{32} + 3u_3 , \\
 D_\infty &= u_0 + u_1 + u_2 + 2w_{03} + 2w_{13} + 2w_{23} + 3v_3
\end{align*}
and $w_{00}$.

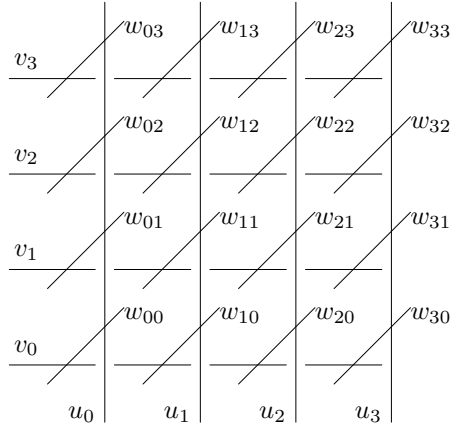
\begin{figure}
\unitlength 0.1in
\begin{picture}(25.000000,24.000000)(-1.000000,-23.500000)
\put(4.500000, -22.000000){\makebox(0,0)[rb]{$u_0$}}%
\put(9.500000, -22.000000){\makebox(0,0)[rb]{$u_1$}}%
\put(14.500000, -22.000000){\makebox(0,0)[rb]{$u_2$}}%
\put(19.500000, -22.000000){\makebox(0,0)[rb]{$u_3$}}%
\put(0.250000, -18.500000){\makebox(0,0)[lb]{$v_0$}}%
\put(0.250000, -13.500000){\makebox(0,0)[lb]{$v_1$}}%
\put(0.250000, -8.500000){\makebox(0,0)[lb]{$v_2$}}%
\put(0.250000, -3.500000){\makebox(0,0)[lb]{$v_3$}}%
\put(6.000000, -16.000000){\makebox(0,0)[lt]{$w_{00}$}}%
\put(6.000000, -11.000000){\makebox(0,0)[lt]{$w_{01}$}}%
\put(6.000000, -6.000000){\makebox(0,0)[lt]{$w_{02}$}}%
\put(6.000000, -1.000000){\makebox(0,0)[lt]{$w_{03}$}}%
\put(11.000000, -16.000000){\makebox(0,0)[lt]{$w_{10}$}}%
\put(11.000000, -11.000000){\makebox(0,0)[lt]{$w_{11}$}}%
\put(11.000000, -6.000000){\makebox(0,0)[lt]{$w_{12}$}}%
\put(11.000000, -1.000000){\makebox(0,0)[lt]{$w_{13}$}}%
\put(16.000000, -16.000000){\makebox(0,0)[lt]{$w_{20}$}}%
\put(16.000000, -11.000000){\makebox(0,0)[lt]{$w_{21}$}}%
\put(16.000000, -6.000000){\makebox(0,0)[lt]{$w_{22}$}}%
\put(16.000000, -1.000000){\makebox(0,0)[lt]{$w_{23}$}}%
\put(21.000000, -16.000000){\makebox(0,0)[lt]{$w_{30}$}}%
\put(21.000000, -11.000000){\makebox(0,0)[lt]{$w_{31}$}}%
\put(21.000000, -6.000000){\makebox(0,0)[lt]{$w_{32}$}}%
\put(21.000000, -1.000000){\makebox(0,0)[lt]{$w_{33}$}}%
\special{pa 500 2200}%
\special{pa 500 0}%
\special{fp}%
\special{pa 1000 2200}%
\special{pa 1000 0}%
\special{fp}%
\special{pa 1500 2200}%
\special{pa 1500 0}%
\special{fp}%
\special{pa 2000 2200}%
\special{pa 2000 0}%
\special{fp}%
\special{pa 0 1900}%
\special{pa 450 1900}%
\special{fp}%
\special{pa 550 1900}%
\special{pa 950 1900}%
\special{fp}%
\special{pa 1050 1900}%
\special{pa 1450 1900}%
\special{fp}%
\special{pa 1550 1900}%
\special{pa 1950 1900}%
\special{fp}%
\special{pa 0 1400}%
\special{pa 450 1400}%
\special{fp}%
\special{pa 550 1400}%
\special{pa 950 1400}%
\special{fp}%
\special{pa 1050 1400}%
\special{pa 1450 1400}%
\special{fp}%
\special{pa 1550 1400}%
\special{pa 1950 1400}%
\special{fp}%
\special{pa 0 900}%
\special{pa 450 900}%
\special{fp}%
\special{pa 550 900}%
\special{pa 950 900}%
\special{fp}%
\special{pa 1050 900}%
\special{pa 1450 900}%
\special{fp}%
\special{pa 1550 900}%
\special{pa 1950 900}%
\special{fp}%
\special{pa 0 400}%
\special{pa 450 400}%
\special{fp}%
\special{pa 550 400}%
\special{pa 950 400}%
\special{fp}%
\special{pa 1050 400}%
\special{pa 1450 400}%
\special{fp}%
\special{pa 1550 400}%
\special{pa 1950 400}%
\special{fp}%
\special{pa 200 2000}%
\special{pa 600 1600}%
\special{fp}%
\special{pa 200 1500}%
\special{pa 600 1100}%
\special{fp}%
\special{pa 200 1000}%
\special{pa 600 600}%
\special{fp}%
\special{pa 200 500}%
\special{pa 600 100}%
\special{fp}%
\special{pa 700 2000}%
\special{pa 1100 1600}%
\special{fp}%
\special{pa 700 1500}%
\special{pa 1100 1100}%
\special{fp}%
\special{pa 700 1000}%
\special{pa 1100 600}%
\special{fp}%
\special{pa 700 500}%
\special{pa 1100 100}%
\special{fp}%
\special{pa 1200 2000}%
\special{pa 1600 1600}%
\special{fp}%
\special{pa 1200 1500}%
\special{pa 1600 1100}%
\special{fp}%
\special{pa 1200 1000}%
\special{pa 1600 600}%
\special{fp}%
\special{pa 1200 500}%
\special{pa 1600 100}%
\special{fp}%
\special{pa 1700 2000}%
\special{pa 2100 1600}%
\special{fp}%
\special{pa 1700 1500}%
\special{pa 2100 1100}%
\special{fp}%
\special{pa 1700 1000}%
\special{pa 2100 600}%
\special{fp}%
\special{pa 1700 500}%
\special{pa 2100 100}%
\special{fp}%
\end{picture}%
 \caption{curves $u_i$, $v_j$ and $w_{ij}$}
 \label{fig:uvw}
\end{figure}

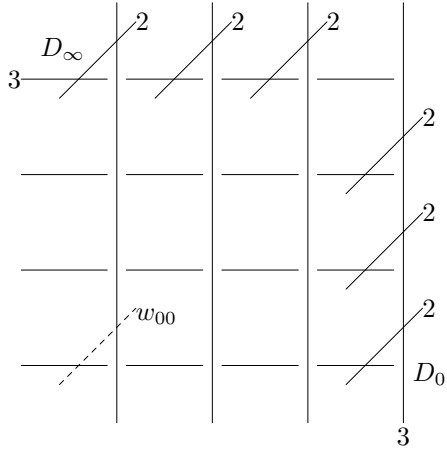
\begin{figure}
\unitlength 0.1in
\begin{picture}(25.000000,24.000000)(-1.000000,-23.500000)
\put(6.000000, -16.000000){\makebox(0,0)[lt]{$w_{00}$}}%
\put(20.500000, -20.000000){\makebox(0,0)[lb]{${D_0}$}}%
\put(1.000000, -3.000000){\makebox(0,0)[lb]{${D_\infty}$}}%
\put(20.000000, -22.250000){\makebox(0,0)[t]{$3$}}%
\put(21.000000, -16.000000){\makebox(0,0)[l]{$2$}}%
\put(21.000000, -11.000000){\makebox(0,0)[l]{$2$}}%
\put(21.000000, -6.000000){\makebox(0,0)[l]{$2$}}%
\put(0.000000, -4.000000){\makebox(0,0)[r]{$3$}}%
\put(6.000000, -1.000000){\makebox(0,0)[l]{$2$}}%
\put(11.000000, -1.000000){\makebox(0,0)[l]{$2$}}%
\put(16.000000, -1.000000){\makebox(0,0)[l]{$2$}}%
\special{pa 500 2200}%
\special{pa 500 0}%
\special{fp}%
\special{pa 1000 2200}%
\special{pa 1000 0}%
\special{fp}%
\special{pa 1500 2200}%
\special{pa 1500 0}%
\special{fp}%
\special{pa 2000 2200}%
\special{pa 2000 0}%
\special{fp}%
\special{pa 0 1900}%
\special{pa 450 1900}%
\special{fp}%
\special{pa 550 1900}%
\special{pa 950 1900}%
\special{fp}%
\special{pa 1050 1900}%
\special{pa 1450 1900}%
\special{fp}%
\special{pa 1550 1900}%
\special{pa 1950 1900}%
\special{fp}%
\special{pa 0 1400}%
\special{pa 450 1400}%
\special{fp}%
\special{pa 550 1400}%
\special{pa 950 1400}%
\special{fp}%
\special{pa 1050 1400}%
\special{pa 1450 1400}%
\special{fp}%
\special{pa 1550 1400}%
\special{pa 1950 1400}%
\special{fp}%
\special{pa 0 900}%
\special{pa 450 900}%
\special{fp}%
\special{pa 550 900}%
\special{pa 950 900}%
\special{fp}%
\special{pa 1050 900}%
\special{pa 1450 900}%
\special{fp}%
\special{pa 1550 900}%
\special{pa 1950 900}%
\special{fp}%
\special{pa 0 400}%
\special{pa 450 400}%
\special{fp}%
\special{pa 550 400}%
\special{pa 950 400}%
\special{fp}%
\special{pa 1050 400}%
\special{pa 1450 400}%
\special{fp}%
\special{pa 1550 400}%
\special{pa 1950 400}%
\special{fp}%
\special{pa 200 500}%
\special{pa 600 100}%
\special{fp}%
\special{pa 700 500}%
\special{pa 1100 100}%
\special{fp}%
\special{pa 1200 500}%
\special{pa 1600 100}%
\special{fp}%
\special{pa 1700 2000}%
\special{pa 2100 1600}%
\special{fp}%
\special{pa 1700 1500}%
\special{pa 2100 1100}%
\special{fp}%
\special{pa 1700 1000}%
\special{pa 2100 600}%
\special{fp}%
\special{pa 200 2000}%
\special{pa 600 1600}%
\special{da 0.03}%
\end{picture}%
 \caption{divisors $D_0$, $D_\infty$ and $w_{00}$}
 \label{fig:d0dinf}
\end{figure}

It is easily seen, 
from the configuration of these divisors displayed in Fig.~\ref{fig:d0dinf}, 
that $D_0$ and $D_\infty$ are disjoint divisors of type $\rIV^*$
with $w_{00} \cdot D_0 = w_{00} \cdot D_\infty = 1$.
Then by Proposition~\ref{prop:makeellfib}
there exists an elliptic fibration $\Phi_X \colon X \rightarrow \bP^1$
having $D_0$ and $D_\infty$ as singular fibers and $w_{00}$ as the image of a section.

Define involutions $\iota_1$, $\iota_2$ on $X$ as follows.
The multiplication-by-$(-1)$ map on the generic fiber $X_\eta$ 
(regarded as an elliptic curve over $\eta = \Spec K(\bP^1)$, 
the origin given by $w_{00}$)
induces an involution $\iota_1$ on $X$, 
which acts on each fiber also by $(-1)$.
The multiplication-by-$(-1,1)$ (or $(1,-1)$) map on $C_1 \times C_2$
induces an involution $\iota_2$ on $X = \Km (C_1 \times C_2)$.
Put $\iota_X = \iota_1\iota_2$.
\begin{claim}
This automorphism $\iota_X$ is a symplectic involution and 
the minimal desingularization of $X / \langle \iota_X \rangle$ is a K3 surface isomorphic to $Y$.
\end{claim}

\begin{proof}
We look at the defining equation (\textasteriskcentered). 
(By the uniqueness of the minimal smooth model, 
we can forget about desingularization 
and consider only generic equations.)
Letting 
\[
u = (t+b'/t) \quad \text{and} \quad
w = (t-b'/t)^{-1} y,
\]
we see that $X$ is defined by the equation
\[
 (u^2-4b') w^2 = x^3 + a x + (b_0 + b_1 u),
\]
which indicates two elliptic fibration structure: 
one over $\bP^1$ with coordinate $u$, 
with singular fibers of type
 $\{ \rII^*, \rI^*_c, \rI^*_{c'} \}$ or $\{ \rII^*, \rI^*_0, \rIV^* \}$, 
and another over $\bP^1$ with coordinate $w$.
Letting
$
 v = u w + b_1 / 2w
$,
we obtain a Weierstrass equation 
\[
 v^2 = x^3 + a x + (b_0 + 4b' w^2 + \frac{b_1^2}{4} w^{-2})
\]
relative to the latter fibration, with two singular fibers of type $\rIV^*$.
Then, by the explicit calculation of Kuwata--Shioda~\cite[sections 2.2 and 5.3]{Kuwata--Shioda}), 
we see that this fibration coincides with our $(D_0, D_\infty)$-fibration 
and that the involution $\iota_X$ acts on this equation by $(w,x,v) \mapsto (-w,x,v)$.
Then the quotient $X / \iota_X$ is birational to $Y$.
\end{proof}

We now describe the fixed points of $\iota_X$ explicitly.
Let $P_{ij}$ and $Q_{ij}$ ($i, j \in \{0,1,2,3\}$) 
be the intersection of $w_{ij}$ with $u_i$ and $v_j$ respectively.
Let $\phi$ (resp.\ $\psi$) be the involution of $u_3$ (resp.\ $v_3$)
which fixes $P_{30}$ (resp.\ $Q_{03}$)
and interchanges $P_{31}$ and $P_{32}$ (resp.\ $Q_{13}$ and $Q_{23}$) 
(such an involution is unique).
Denote by $P_{3\infty}$ (resp.\ $Q_{\infty3}$)
the fixed point of $\phi$ (resp.\ $\psi$) other than $P_{30}$ (resp.\ $Q_{03}$).
\begin{claim}
The fixed points of $\iota_X$ are $P_{00}$, $Q_{00}$, $P_{03}$, $Q_{03}$, $P_{30}$, $Q_{30}$, $P_{3\infty}$ and $Q_{\infty3}$.
\end{claim}
\begin{proof}
In order to show this, 
by Lemma~\ref{lem:nikulin-inv}, 
we only have to show that $\iota_X$ indeed fixes these eight points. 
We will show that both $\iota_1$ and $\iota_2$ fix these points.

It is clear that $\iota_2$ fixes each $u_i$ and $v_j$ pointwise.
Hence $\iota_2$ in particular fixes each $P_{ij}$, $Q_{ij}$, $P_{3\infty}$ and $Q_{\infty3}$.

By construction $\iota_1$ fixes $w_{00}$ pointwise and hence $P_{00}$ and $Q_{00}$. 
Since $\iota_1$ also fixes each fiber (not pointwise),
$\iota_1$ fixes the components $u_0$, $w_{03}$, $v_3$, 
and similarly $v_0$, $w_{30}$, $u_3$ (all not pointwise).
Hence $\iota_1$ fixes $P_{03}$, $Q_{03}$, $P_{30}$ and $Q_{30}$.
As $\iota_1$ acts by $-1$ on the group scheme $(D_\infty)^\sm$
 (which is the disjoint union of three components each isomorphic to $\bG_a$), 
$\iota_1$ interchanges $u_1$ and $u_2$, hence $w_{13}$ and $w_{23}$, and hence $Q_{13}$ and $Q_{23}$.
This means that $\iota_1$ acts on $v_3$ by $\psi$. Hence it fixes $Q_{\infty3}$. 
Similarly it fixes $P_{3\infty}$.
\end{proof}

Let $\widetilde X$ be the blow-up of $X$ at these eight points
and $\tilde \iota_X$ the involution of $\widetilde X$ induced by $\iota_X$.
Then one can easily check that $Y$, 
which is isomorphic to the minimal desingularization of $X / \langle \iota_X \rangle$,
is also isomorphic to $\widetilde X / \langle \tilde \iota_X \rangle$. 

\medskip

We will now extend this construction to the relative case (over $\cO_K$). 
By the construction of $\cX$
and the fact that the abelian surface $\cA$ is the product of two elliptic curves (over $\cO_K$),
the $24$ rational curves on $X$ extends naturally to 
closed subschemes on $\cX$ which are each isomorphic to $\bP^1_{\cO_K}$.
Using these subschemes, we define divisors $\cD_0$, $\cD_\infty$ and $\cW_{00}$
similarly as $D_0$, $D_\infty$ and $w_{00}$.
Also we define the ``points'' $\cP_{ij}$, $\cQ_{ij}$, $\cP_{3\infty}$ and $\cQ_{\infty3}$ 
similarly as $P_{ij}$, $Q_{ij}$, $P_{3\infty}$ and $Q_{\infty3}$
(these are closed subschemes each isomorphic to $\Spec \cO_K$).

Hereafter, 
we denote schemes over $\cO_K$ by calligraphic letters (e.g.\ $\cC$) 
and their generic and special fibers by italic letters equipped with suffixes ${}_K$ and ${}_k$ (e.g.\ $C_K$ and $C_k$).
For sheaves on $\cO_K$-schemes or morphisms of $\cO_K$-schemes,
we denote their restrictions to the generic and special fibers by the same letter with suffixes ${}_K$ and ${}_k$
(e.g.\ $\Phi_K$ and $\Phi_k$ are the restrictions of $\Phi$).

We use the next proposition, 
which is a relative version of Proposition~\ref{prop:makeellfib}. 

\begin{proposition}
 \label{prop:makeellfibrel}
Let $\cX$ be a smooth proper scheme over $\cO_K$ 
such that $X_K$ and $X_k$ are K3 surfaces respectively over $K$ and $k$. 
Let $\cZ$ and $\cC_i$ be subschemes of $\cX$
and let $\cD = \sum n_i \cC_i$ be a (finite) linear combination.
Assume that 
\begin{itemize}
 \item each $\cC_i$ and $\cZ$ is isomorphic to $\POK^1$, 
 \item the intersection of $\cZ$ and $\cD$ is a scheme isomorphic to $\Spec \cO_K$,  
and 
 \item $\cD = \sum n_i \cC_i$ is a configuration of the type 
  $\rI_n$ ($n \geq 2$), $\rI^*_n$ ($n \geq 0$), $\rII^*$, $\rIII^{(*)}$ or $\rIV^{(*)}$ in Kodaira's notation%
  \footnote{Of course the intersection of two components should be 
  the spectrum of a ring ($\cO_K$) instead of the spectrum of a field.
  Here we exclude types $\rI_0$, $\rI_1$ and $\rII$ because their components are not $\bP^1$.}.
\end{itemize}
(It then follows that $D_K$ and $D_k$ satisfy conditions of Proposition~{\rm \ref{prop:makeellfib}}.) 

Then there exists an ``elliptic fibration'' $\Phi \colon \cX \rightarrow \bP^1_{\cO_K}$ 
having $\cD$ as a ``singular fiber'' and $\cZ$ as the image of a section, 
that is, $\Phi$ satisfies the following:
\begin{itemize}
 \item $\Phi$ is a proper surjection. 
 \item $\Phi_K \colon X_K \rightarrow \bP^1_K$
  and $\Phi_k \colon X_k \rightarrow \bP^1_k$
  are elliptic fibrations in the usual meaning.
 \item The composite $\cZ \hookrightarrow \cX \rightarrow \POK^1$ is an isomorphism.
 \item There exists an $\cO_K$-valued point $s \in \POK^1(\cO_K)$
  such that $\Phi^{-1}(s) \allowbreak = \cD$.
\end{itemize}

Moreover if $\cD'$ is as in Proposition~{\rm \ref{prop:makeellfib} (2)} then $\cD'$ is another ``singular fiber''.
\end{proposition}

To prove this, we need a well-known lemma on cohomology of fibers. 
For a proof see \cite[Theorem 5.3.20]{Liu:AGAC}.

\begin{lemma}
Let $\cX$ be a proper $\cO_K$-scheme
and $\sF$ a coherent sheaf on $\cX$, flat over $\cO_K$.
Then 
\begin{enumerate}
 \item $\dim_k H^p(X_k, \sF_k) \geq \dim_K H^p(X_K, \sF_K)$, 
 \item the equality holds if and only if 
  $H^p(\cX, \sF)$ is a free $\cO_K$-module such that
  the canonical morphism $H^p(\cX, \sF) \otimes_{\cO_K} k \rightarrow H^p(X_k, \sF_k)$ is an isomorphism, and
 \item the morphism in {\rm (2)} is an isomorphism if and only if $H^{p+1}(\cX, \sF)$ is a free $\cO_K$-module.
\end{enumerate}
\end{lemma}

\begin{proof}
 [Proposition~{\rm \ref{prop:makeellfibrel}}]
First we show that $H^0(\cX, \cO_\cX(\cD)) \cong \cO_K^{\oplus 2}$.

We make use of the cohomology long exact sequence of the sequence
\[
 0 \rightarrow \cO_\cX \rightarrow \cO_\cX(\cD) \rightarrow \cO_X(\cD) \otimes \cO_\cD \rightarrow 0 .
\]
First note that $\cO_\cX(\cD) | _\cD \cong \cO_\cD$: 
it is true on the generic fiber since $(C_i)_K \cdot D_K = 0$ for each component $\cC_i$ of $\cD$, and
there is a canonical isomorphism $\Pic \cC_i \cong \Pic (C_i)_K$.
Connectedness of $X_K$ and $X_k$ yields $H^0(X_K, \cO_{X_K}) = K$ and $H^0(X_k, \cO_{X_k}) = k$.
Hence by the lemma we have $H^0(\cX, \cO_{\cX}) = \cO_K$, 
and again by the lemma $H^1(\cX, \cO_{\cX})$ is free over $\cO_K$.
Since cohomology commutes with taking the generic fiber (which is a flat base change), 
$H^1(\cX, \cO_\cX) \otimes_{\cO_K} K = H^1(X_K, \cO_{X_K}) = 0$ and hence $H^1(\cX, \cO_\cX) = 0$. 
Similarly, $H^0(\cD, \cO_{\cD}) = \cO_K$, 
and $H^1(\cD, \cO_{\cD})$ is free over $\cO_K$.
Combining these information, we obtain an exact sequence
\[
\begin{array}{ccccccc}
0 &\rightarrow& \cO_K &\rightarrow& H^0(\cX, \cO_\cX(\cD)) &\rightarrow& \cO_K \\
  &\rightarrow& 0 &\rightarrow& H^1(\cX, \cO_\cX(\cD)) &\rightarrow& H^1(\cD, \cO_\cD).
\end{array}
\]
It follows that $H^0(\cX, \cO_\cX(\cD)) \cong \cO_K^{\oplus 2}$.
So this ``linear system'' defines a morphism $\Phi \colon \cX \rightarrow \POK^1$.

Next we will show that this construction is compatible with that of Proposition~\ref{prop:makeellfib} (1), 
that is, $\Phi_K \colon X_K \rightarrow \bP^1_K$ and 
$\Phi_k \colon X_k \rightarrow \bP^1_k$ is the same as 
those constructed in Proposition~\ref{prop:makeellfib} (1).
This will show that $\Phi$ is the morphism wanted in the proposition.

Again by the compatibility with flat base change, we have $H^0(X_K, \cO_{X_K}(D_K)) \cong H^0(\cX, \cO_\cX(\cD)) \otimes_{\cO_K} K$. 
For the special fiber, 
$H^1(\cX, \cO_\cX(\cD))$ is a free $\cO_K$-module (since it is a submodule of a free module $H^1(\cD, \cO_\cD)$), 
so by the lemma we have $H^0(X_k, \cO_{X_k}(D_k)) \cong H^0(\cX, \cO_\cX(\cD)) \otimes_{\cO_K} k$.
These equalities show that $\Phi_K$ and $\Phi_k$ are those constructed in Proposition~\ref{prop:makeellfib} (1).

The last assertion is proved by following the proof of Proposition~\ref{prop:makeellfib} (2).
\end{proof}

We return to the proof of Theorem~\ref{thm:critetSI}.
Proposition~\ref{prop:makeellfibrel} shows that there exists an ``elliptic fibration''
$\Phi_\cX \colon \cX \rightarrow \bP^1_{\cO_K}$. 
One defines $\iota_\cX \colon \cX \rightarrow \cX$ similarly 
and observes (following the proof of the previous Claim) that the fixed points are (the union of)
$\cP_{00}$, $\cQ_{00}$, $\cP_{03}$, $\cQ_{03}$, $\cP_{30}$, $\cQ_{30}$, $\cP_{3\infty}$ and $\cQ_{\infty3}$.
Let $\cY = \widetilde \cX / \langle \tilde \iota_\cX \rangle $ where 
$\widetilde \cX$ is the blow-up of $\cX$ at the (union of) fixed points
and $\tilde \iota_\cX$ is the involution on $\widetilde \cX$ induced by $\iota_\cX$.
We shall show that this $\cY$ is a smooth proper model of $Y$.

Properness and flatness of $\cY$ over $\cO_K$ is clear from the construction.
We also know that 
$Y$ and $Y'$ are nonsingular, 
where we denote by $Y'$ the surface obtained by performing similar operations 
on the special fiber $X_k$ of $\cX$.
Hence it suffices to check that the generic fiber and special fiber of $\cY$ are 
isomorphic to $Y$ and $Y'$ respectively.
Since we have assumed that the residue characteristic is not equal to the order of $\iota_\cX$ (=$2$), 
$(\widetilde \cX / \langle \tilde \iota_\cX \rangle) \times_{\cO_K} k$ is isomorphic to 
$(\widetilde \cX \times_{\cO_K} k) / \langle \tilde \iota_k \rangle$.
Since this blow-up commutes with base change by Lemma~\ref{lem:blupatsmpt},
we have $\widetilde \cX \times_{\cO_K} k \cong (X_k)^\sim$ 
and hence $(\widetilde \cX \times_{\cO_K} k) / \langle \tilde \iota_k \rangle 
\cong (X_k)^\sim / \langle \tilde \iota_k \rangle \cong Y'$.
The generic case is easier 
(since blow-up always commutes with flat base change, we do not need Lemma~\ref{lem:blupatsmpt}).

This concludes the proof of Theorem~\ref{thm:critetSI}.

\begin{remark}
 \label{rem:2morph}
In the proof, we used two different rational maps 
$Y_{\overline K} \rightarrow X_{\overline K}$ and $X_{\overline K} \rightarrow Y_{\overline K}$.
Here we explain why we needed both.

Let us try to construct $X$ from $Y$ via the map $X_{\overline K} \rightarrow Y_{\overline K}$.
We can determine the branch locus of $X_{\overline K} \rightarrow Y_{\overline K}$ explicitly 
(it is the union of the components of odd multiplicity in two fibers of type $\rII^*$) and so
we can define $X$ and $X \rightarrow Y$ over a finite unramified extension of $K$. 
However, for the relationship of their cohomologies, we merely obtain
\[
 H^2_\et(Y_{\overline K}, \bQ_l) \cong
  H^2_\et(X_{\overline K}, \bQ_l)^{\langle \iota \rangle} \oplus \bigoplus \bQ_l(-1) [E_i]
\]
and we cannot deduce that $H^2_\et(X_{\overline K}, \bQ_l)$ is unramified, even after taking some (ramified) extension on $K$.

Next let us try to construct $\cY$ from $\cX$ via the map $Y_{\overline K} \rightarrow X_{\overline K}$.
According to the construction of Shioda--Inose~\cite[Section 2]{Shioda--Inose}, 
$X$ admits an elliptic fibration with (at least) three singular fibers of types 
$\{ \rII^*, \rI^*_c, \rI^*_{c'} \}$ or
$\{ \rII^*, \rI^*_0, \rIV^* \}$, 
and the branch locus of the morphism $Y \rightarrow X$ is the union of the components of multiplicity $1$ of 
fibers of type $\rI^*_c$ or $\rIV^*$. 
Unfortunately, the types of singular fibers might be different between the generic and special fiber of $\cX$, 
and we have trouble constructing $\cY$ as a double cover of $\cX$. 
\end{remark}

\begin{remark}
We can give an explicit (but far from best possible) bound for the \emph{separable} degree of the extension needed.

Checking the proof of the theorem and of Theorem~\ref{thm:critetKum}, 
we see that the only places we need field extensions are
(i) where we use Lemmas \ref{lem:-2curves} (twice) and \ref{lem:prodabelsurf}, 
and
(ii) where we take the kernel of the permutation action on $8$ fixed points.
The degree of extension in (ii) has trivial bound $8!$.
For each time in (i), 
it suffices to take an extension $K'$ so that 
$G_{K'}$ acts trivially on $\NS(W_{\overline K})$
where $W$ is one of $Y$, $X$, $A$.
The same arguments as in Remark~\ref{rem:degboundKum}
give explicit bounds.

Combining these, 
we have a bound $3^{484 + 484 + 36} \cdot 8! \leq 10^{484}$.

If $\charac K > 0$, this argument remains valid
by Lemma~\ref{lem:tower}.
However we do not have any bound for the \emph{inseparable} degree.
\end{remark}

We conclude this section with the proof of Corollary~\ref{cor:singk3}. 

\begin{proof}[Corollary~{\rm \ref{cor:singk3}}]
Any singular K3 surface has Shioda--Inose structure of product type
such that corresponding elliptic curves $C_1$, $C_2$ have complex multiplication
 (Shioda--Inose~\cite[Theorem 4]{Shioda--Inose}).
Any elliptic curve with complex multiplication 
is defined and has good reduction over some number field. 
Using the construction of $\cY$ from $\cX$ above, the corollary follows.
\end{proof}

\section{$p$-adic criterion} \label{sec:pproof}

We now focus on $p$-adic cohomology.
As we remarked before, in this section $K$ is of mixed characteristic $(0,p)$, 
with $k$ perfect. 
We denote by $K_0$ the unramified closure of $\bQ_p$ in $K$.

We first overview the proofs of Theorems \ref{thm:critcrysKum} and \ref{thm:critcrysSI}.
As in the $l$-adic case, 
the main idea for the Kummer case (resp.~Shioda--Inose case)
is to reduce to the abelian case (resp.~Kummer case).
However, there are some more difficulties, 
to overcome which we need our $l$-adic results.
For Theorem~\ref{thm:critcrysKum}, given a Kummer surface $X$ as in the theorem, 
we 
\begin{enumerate}
 \item[(Km1)]
  construct an abelian surface $A$ corresponding to it
  over a finite extension of $K$, 
 \item[(Km2)] 
  (using the crystallineness hypothesis) 
  show that $H^2_\et(A_{\overline K}, \bQ_p)$ is crystalline 
  after taking a finite unramified extension of $K$
  and a quadratic twist of $A$,
 \item[(Km3)] 
 hence obtain a good model $\cA$ 
  (by Coleman--Iovita~\cite[Theorem 4.7]{Coleman--Iovita}), 
 \item[(Km4)] 
 construct a good model $\cX$
  using the rational map $A_{\overline K} \rightarrow X_{\overline K}$,
  and
 \item[(Km5)] 
  show that we can take the extension in step (Km1) to be unramified.
\end{enumerate}
Similarly, 
for Theorem~\ref{thm:critcrysSI}, 
given a K3 surface $Y$ admitting a Shioda--Inose structure of product type as in the theorem, 
we 
\begin{enumerate}
 \item[(SI1)]
  construct a Kummer surface $X$ corresponding to it
  over a finite extension of $K$, 
 \item[(SI2)]
  (using the crystallineness hypothesis) 
  show that $H^2_\et(X_{\overline K}, \bQ_p)$ is crystalline 
  after taking a finite extension of $K$ of ramification index $1$, $2$, $3$, $4$ or $6$, 
 \item[(SI3)]
  hence by Theorem~\ref{thm:critcrysKum}, after some finite unramified extension, obtain a good model $\cX$,
 \item[(SI3$'$)] which we may assume (after further finite unramified extension) to be obtained from 
 the product of two elliptic curves and hence have $24$ specific smooth ``rational curves'',
 \item[(SI4)]
  construct a good model $\cY$
  using the rational map $X_{\overline K} \rightarrow Y_{\overline K}$,
  and
 \item[(SI5)]
  show that we can take the extension in step (SI1) to be unramified.
\end{enumerate}

\medskip

Steps (Km1) and (SI1) are easy.
As in the $l$-adic case, 
it suffices to take an extension over which certain (finitely many) curves are rational.
(Note that we do not, at this moment, require the extension to be unramified.)

Step (Km2):
Let $X$ be a Kummer surface such that $H^2_\et(X_{\overline K}, \bQ_p)$ is crystalline and
$A$ an abelian surface over $K$ 
such that $\Km A \cong X$.
We show that (after taking some finite unramified extension of $K$) there exists an abelian surface $A'$ such that $\Km A' \cong \Km A$ and 
$H^1_\et(A'_{\overline K}, \bQ_p)$ is semi-stable, 
and then show that it is automatically crystalline.

There exists a finite Galois extension $L/K$ such that $V = H^1_\et(A_{\overline K}, \bQ_p)$ is a semi-stable representation of $G_L$.
By replacing $K$ by its unramified closure in $L$, 
we may assume that $L/K$ is totally ramified.
Put $D = D_{\st,L}(V) = (B_\st \otimes V)^{G_L}$. 
Since $D_{\st,L}$ commutes with exterior product for semi-stable representations of $G_L$, we have
\[ 
 D_{\st,K}(\bigwedge^2 V) = \bigl( D_{\st,L}(\bigwedge^2 V) \bigr)^G = (\bigwedge^2 D)^G
\]
where $G = \Gal(L/K)$.
Since 
$\bigwedge^2 V = H^2_\et(A_{\overline K}, \bQ_p)$ is a crystalline (hence semi-stable) representation of $G_K$ by assumption, 
and since $V$ is semi-stable representation of $G_L$, 
we have 
\[
 \dim_{K_0} \bigl( \bigwedge^2 D \bigr)^G = \dim_{\bQ_p} \bigwedge^2 V = \dim_{L_0} \bigwedge^2 D.
\]
Since $L_0 = K_0$, it follows that $G$ acts on $\bigwedge^2 D$ trivially.
Then by the next lemma, 
there exists a subgroup $G' \subset G$ of index at most $2$
such that $G'$ acts on $D$ trivially. 

\begin{lemma}
 \label{lem:wedge2}
Let $W$ be a vector space over a field 
with $\dim W \geq 3$
and $f$ a linear automorphism of $W$.
If $\bigwedge^2 f$ acts as the identity on $\bigwedge^2 W$ then
$f$ is either the identity or $(-1)$ times the identity.
\end{lemma}
\begin{proof}
This is an easy exercise of linear algebra. 
\end{proof}

If $G' = G$ then $H^1_\et(A_{\overline K}, \bQ_p)$ is already semi-stable, so we can take $A' = A$.
Assume $G' \subsetneq G$.
We follow the construction in the proof of Theorem~\ref{athm:Kum}.
Let $M$ be the quadratic extension of $K$ corresponding to $G'$.
Then we have $\dim_{K_0} D^{G'} = \dim_{K_0} D = \dim_{\bQ_p} V$, 
which means that $V$ is semi-stable as a representation of $G_M = G'$.
Put $G'' = G/G' = \Gal(M/K)$.
Let $G'' \cong \{ \pm 1 \}$ act on $A$ by $\pm 1$ 
and consider another abelian surface $A' = (A \times_K M)/G''$ over $K$, 
where $G''$ acts diagonally on $A \times_K M$. 
It is clear from the construction that $\Km A \cong \Km A'$.
The surfaces $A'_M$ and $A_M$ are naturally isomorphic 
and hence $H^1_\et(A'_{\overline K}, \bQ_p) \cong H^1_\et(A_{\overline K}, \bQ_p)$ 
as a representation of $G_M$.
The action of $g \in G_K$ on $H^1_\et(A'_{\overline K}, \bQ_p)$ 
is equal to $\pm 1$ times the action of $g$ on $H^1_\et(A_{\overline K}, \bQ_p)$ 
where the sign is positive if $g \in G_M$ and negative otherwise.

Put $V' = H^1_\et(A'_{\overline K}, \bQ_p)$ and 
$D' = D_{\st,L}(V') = (V' \otimes B_\st)^{G_L}$.
Then as in the proof of Theorem~\ref{athm:Kum}, 
$D_{\st,K}(V') = (D')^{G_K}$ has appropriate dimension over $K_0$. 
Thus $V'$ is a semi-stable representation of $G_K$.

It remains to show that the monodromy $N$ of $D_{\st,K}(V')$ is zero.
From the crystallineness hypothesis the monodromy of $D_{\st,K}(\bigwedge^2 V')$ is zero, 
that is, $N \wedge 1 + 1 \wedge N = 0$ on $\bigwedge^2 D_{\st,K}(V')$. 
By applying Lemma~\ref{lem:wedge2} to $\exp N = \sum N^k/k!$ (which is a finite sum since $N$ is nilpotent)
we obtain $\exp N = \pm 1$ and hence $N = 0$.
This means $V'$ is crystalline, 
thus concludes step (Km2).

\medskip

Step (SI2):
Assume we are given surfaces $Y$ and $X$
where $X$ is the minimal desingularization of $Y / \langle \iota \rangle$.
As in the $l$-adic case we have 
\[
 H^2_\et(X_{\overline K}, \bQ_p) 
 \cong H^2_\et(Y_{\overline K}, \bQ_p)^{\langle \iota \rangle} \oplus \bigoplus_{i=1}^8 \bQ_p(-1) [E_i] 
\]
where the last term is the Tate twist of the permutation representation 
corresponding to the eight $\iota$-fixed points.
Let $H \subset G_K$ be (as before) the kernel of this permutation action
and $K_H/K$ the corresponding (finite) extension, 
which is of ramification index $1$, $2$, $3$, $4$ or $6$.
Then $\bigoplus \bQ_p(-1) [E_i]$, being the Tate twist of a trivial representation, 
is a crystalline representation of $G_{K_H}$.
Also $H^2_\et(Y_{\overline K}, \bQ_p)^{\langle \iota \rangle}$, 
being a direct summand of a crystalline representation $H^2_\et(Y_{\overline K}, \bQ_p)$, is crystalline.
So step (SI2) is done.

\medskip

Steps (Km3) and (SI3) are just applying the indicated results.

For step (SI3$'$), 
we need a $p$-adic analogue of Lemma~\ref{lem:prodabelsurf}.
But we can reduce this to ($l$-adic) Lemma~\ref{lem:prodabelsurf}:
if $H^1_\et(A_{\overline K}, \bQ_p)$ is crystalline
then $A$ has good reduction and hence $H^1_\et(A_{\overline K}, \bQ_l)$ is unramified.

Steps (Km4) and (SI4) are the same as in the $l$-adic case.

For steps (Km5) and (SI5), we need a $p$-adic analogue of Lemma~\ref{lem:-2curves}.
The next proposition reduces this to ($l$-adic) Lemma~\ref{lem:-2curves}.
(Note that the potential good reduction assumption is satisfied since 
we have already proved steps (Km1--Km4) and (SI1--SI4).)

\begin{proposition}
 \label{prop:crysunramext}
Let $X$ be a K3 surface over $K$ with potential good reduction.
Assume that $H^2_\et(X_{\overline K}, \bQ_p)$ is crystalline.
Then $H^2_\et(X_{\overline K}, \bQ_l)$ is unramified for any prime $l \neq p$.
\end{proposition}

\begin{proof}
Take a finite Galois extension $K'/K$ such that 
$X_{K'}$ has good reduction.
Then the action of $I_K$ on $H^2_\et(X_{\overline K}, \bQ_l)$ 
factors through $I(K'/K) = I_K/I_{K'}$.
Let $\cX$ be a good model of $X_{K'}$ over $\cO_{K'}$, 
and let $\overline X = \cX \times_{\cO_{K'}} k'$.

Take an arbitrary element $\sigma \in I({K'}/K)$ 
and denote by $\cX^\sigma$ the scheme obtained 
from $\cX$
by the base change 
$\sigma^* \colon \Spec \cO_{K'} \rightarrow \Spec \cO_{K'}$.
Denote by $\Gamma$ the (scheme-theoretic) closure of 
$\Delta(X_{K'})$ in $\cX \times_{\cO_{K'}} \cX^\sigma$
(where $\Delta \colon X_{K'} \rightarrow X_{K'} \times_{K'} X_{K'}$ is the diagonal map).
Since $\cX \times_{\cO_{K'}} \cX^\sigma$ is regular
there exists a resolution $\cE_\bullet \rightarrow \cO_\Gamma$ of finite length
by locally free modules of finite rank.
Put 
\begin{align*}
 \widetilde \Gamma 
  &= \sum_i (-1)^i \ch_2(\cE_i) 
  && \in \CH^2(\cX \times_{\cO_{K'}} \cX^\sigma)_\bQ
\quad \text{and} \\
 \overline \Gamma = \widetilde \Gamma |_{\overline X \times_{k'} \overline X}
  &= \sum_i (-1)^i \ch_2(\cE_i|_{\overline X \times_{k'} \overline X}) 
  && \in \CH^2(\overline X \times_{k'} \overline X)_\bQ
\end{align*}
(these do not depend on the choice of the resolution).

Then by Riemann--Roch (see \cite[Lemma 2.17]{Saito:weightSS})
the restriction of $\widetilde \Gamma$ on the generic fiber 
coincides with $\Delta(X_{K'})$.
Hence we have a commutative diagram 
\[
\xymatrix{
  H^i_\et(X_{\overline K}, \bQ_l)
  \ar[d]_{\sigma_*}
& H^i_\et(\overline X_{\overline {k'}}, \bQ_l)
  \ar[l]_{=}
  \ar[d]_{\overline \Gamma^*}
\\
  H^i_\et(X_{\overline K}, \bQ_l)
& H^i_\et(\overline X_{\overline {k'}}, \bQ_l)
  \ar[l]_{=}
}
\]
(Saito~\cite[Corollary 2.20]{Saito:weightSS}) and hence an equality 
\[
\Tr(\sigma_*  \mid  H^i_\et(X_{\overline K}, \bQ_l))
 = \Tr(\overline \Gamma^*  \mid  H^i_\et(\overline X_{\overline {k'}}, \bQ_l)).
\]

By the Lefschetz trace formula we have
\begin{gather*}
 \sum_i (-1)^i \Tr( \overline \Gamma^*  \mid  H^i_\et(\overline X_{\overline {k'}}, \bQ_l))
=
 (\overline \Gamma, \Delta(\overline X)) 
= 
 \sum_i (-1)^i \Tr( \overline \Gamma^*  \mid  H^i_\crys(\overline X))
\end{gather*}
where the intersection number is taken in $\overline X \times_{k'} \overline X$.

Since the isomorphism in the crystalline conjecture is compatible
with pull-backs, cup products with cycle classes and direct images
(Tsuji~\cite{Tsuji:Cst}, \cite[Theorem A2]{Tsuji:Cstsurvey} 
and Berthelot--Ogus~\cite[Proposition 3.4]{Berthelot--Ogus:dRcrys}), 
it is compatible with the action of a correspondence. 
So we have a commutative diagram
\[
\xymatrix{
  D_{\crys,K'}(H^i_\et(X_{\overline K}, \bQ_p))
  \ar[r]^(0.65){=}
  \ar[d]_{\sigma_*}
& H^i_\crys(\overline X) 
  \ar[d]_{\overline \Gamma^*}
\\
  D_{\crys,K'}(H^i_\et(X_{\overline K}, \bQ_p))
  \ar[r]^(0.65){=}
& H^i_\crys(\overline X)
.
}
\]
Since $H^i_\et(X_{\overline K}, \bQ_p)$ is a crystalline representation for all $i$ 
(for $i=2$ this is the assumption, for $i=0,4$ it is clear and for $i=1,3$ the cohomologies vanish), 
we have 
\[
 D_{\crys,K'}(H^i_\et(X_{\overline K}, \bQ_p)) = D_{\crys,K}(H^i_\et(X_{\overline K}, \bQ_p)) \otimes_{K_0} K'_0
\] 
and hence $I(K'/K)$ acts on $D_{\crys,K'}(H^i_\et(X_{\overline K}, \bQ_p))$ trivially.
By the above diagram, 
$\overline \Gamma^*$ acts on $H^i_\crys (\overline X)$ trivially.
So 
\[ 
  \Tr( \overline \Gamma^*  \mid  H^i_\crys(\overline X))
 = \dim_{K'_0} H^i_\crys(\overline X)
 = \dim_{\bQ_p} H^i_\et(X_{\overline K}, \bQ_p)
.\]

Finally (by comparing both sides with the Betti numbers) we have 
\[ 
  \dim_{\bQ_p} H^i_\et(X_{\overline K}, \bQ_p)
= \dim_{\bQ_l} H^i_\et(X_{\overline K}, \bQ_l)
.\]

Combining these equalities we obtain 
\[
 \sum_i \Tr( \sigma_*  \mid  H^i_\et(X_{\overline K}, \bQ_l))
= 
 \sum_i \dim_{\bQ_l} H^i_\et(X_{\overline K}, \bQ_l)
\]
(note again that $H^1 = H^3 = 0$).
Thus each element in $I(K'/K)$ acts on $H^*_\et(X_{\overline K}, \bQ_l)$ 
by trace equal to the dimension of this $\bQ_l$-vector space.
It then follows that the action of this group is trivial.
\end{proof}

\section{Appendix: Good reduction of Kummer surfaces}
 \label{sec:Kummer}

We record the proof of Theorem~\ref{thm:critetKum} from \cite{Ito:Kummer}.

As in section \ref{sec:lproof}, 
we argue as if $\charac K = 0$.

\medskip

First we review the relation between Kummer surfaces and abelian surfaces.

Let $F$ be a field of characteristic $\neq 2$. 
Let $A$ be an abelian surface over $\overline F$, and $X = \Km(A)$.
Let $G = \{ \id, \iota \}$ where $\iota$ is the multiplication-by-$(-1)$ map of $A$. 
The surface $X$ is, by definition, obtained by blow-up at $16$ singular points of $A/G$.
However we get $X$ from $A$ in another way as follows.

Let $\tilde A$ be the blow up of $A$ at $A[2]$. 
Since $A[2]$ is the fixed points of the action of $G$, 
we can extend the action of $G$ on $\tilde A$.
Then the quotient variety $\tilde A /G$ is naturally isomorphic to $X$.
We have a cartesian diagram 
\[
\xymatrix{
 \tilde A \ar[r] \ar[d] 
 & \tilde A/G \cong X \ar[d]
 \\ 
 A \ar[r]
 & A/G ,
}
\]
where the horizontal maps are the quotient maps and the vertical maps are blow-ups at $16$ points.
Let $Z$ be the exceptional divisor of the blow-up $X \rightarrow A/G$.
This is the union of $16$ curves of self-intersection $-2$.
By construction,  $\tilde A \rightarrow X$ is a double covering whose branch locus is $Z$.

\medskip

We prove the following special case of Theorem~\ref{thm:critetKum}.

\begin{theorem} \label{athm:Kum}
Let $A$ be an abelian surface over $K$ 
and $X = \Km(A)$.
Then $X$ has good reduction if and only if $H^2_\et(X_{\overline K}, \bQ_l)$ is unramified.
\end{theorem}

\begin{lemma} \label{alem:Kum}
Let $A$ be an abelian surface over $K$ 
and $X = \Km(A)$.
If $A$ has good reduction, 
then $X$ has good reduction.
\end{lemma}

\begin{proof}
The N\'eron model $\cA$ of $A$ over $\cO_K$ is an abelian scheme over $\cO_K$ (\cite[Proposition 1.4/2]{BLR:neronmodels}).
By the N\'eron mapping property, 
the multiplication-by-$(-1)$ map on $A$
has a natural extension to an involution $\iota \colon \cA \rightarrow \cA$.
In the special fiber, $\iota$ is the multiplication-by-$(-1)$ map and has exactly $16$ fixed points 
(because the residue characteristic $p$ is not equal to $2$).
Let $\tilde \cA$ be the blow up of $\cA$ at $\cA[2]$: 
this is smooth over $\cO_K$ by Lemma~\ref{lem:blupatsmpt}. 
Then $\iota$ induces an involution $\tilde \iota$ on $\tilde \cA$.
Taking the quotient of $\tilde \cA$ by $\tilde \iota$, 
we obtain a proper smooth model of $X$ over $\cO_K$.
\end{proof}

\begin{proof}[Theorem \ref{athm:Kum}]
Since 
$H^2_\et(A_{\overline K}, \bQ_l)$ is a direct summand of $H^2_\et(X_{\overline K}, \bQ_l)$, 
the inertia group $I_K$ acts trivially on $H^2_\et(A_{\overline K}, \bQ_l)$, which is isomorphic to $\bigwedge^2 H^1_\et(A_{\overline K}, \bQ_l)$.
By Lemma \ref{lem:wedge2}, we have a homomorphism $f \colon I_K \rightarrow \{ \pm1 \}$
which the action of $I_K$ on $H^1_\et(A_{\overline K}, \bQ_l)$ factors through.
If $f$ is trivial we are done by applying Theorem~\ref{thm:critAV} and Lemma~\ref{alem:Kum}. 
Assume that $f$ is not trivial.
The idea is to take a quadratic twist of $A$ to get an abelian surface $A'$ over $K$
such that $A'$ has good reduction over $K$ and that $\Km(A) \cong \Km(A')$.

Let $L$ be a ramified quadratic extension of $K$. 
Since the kernel of $f$ corresponds to the (unique) ramified quadratic extension $L K^\un$ of $K^\un$, 
the homomorphism $\tilde f \colon G_K \twoheadrightarrow \Gal(L/K) \cong \{ \pm 1 \}$ extends $f$.
Let $G = \{\id, -1 \}$ be a group of automorphisms of $A$, 
and fix the unique isomorphism $\Gal(L/K) \cong G$ (thus we have an action of $\Gal(L/K)$ on $A$). 
We take a quotient $A' = (A \times_K L) /\Gal(L/K)$ (where $\Gal(L/K)$ acts diagonally on $A \times_K L$).

We shall see that this $A'$ satisfies the desired conditions.
By the above construction we have $A_L \cong A'_L$, 
and hence $I_L$ acts on $H^1_\et(A'_{\overline K}, \bQ_l)$ trivially.
Since $I_K$ acts on $H^1_\et(A_{\overline K}, \bQ_l)$ by $I_K \rightarrow I_K/I_L \cong \Gal(L/K) \cong \{ \pm 1 \}$, 
and since the involution $-1 \in G$ acts by $-1$, 
it follows that $I_K$ acts on $H^1_\et(A'_{\overline K}, \bQ_l)$ trivially.
Hence $A'$ has good reduction over $\cO_K$ by Theorem~\ref{thm:critAV}.

It is easy to see $\Km(A) \cong \Km(A')$: 
the effect of the quadratic twist vanishes after we take the quotients by $G$.
It remains to apply Lemma~\ref{alem:Kum}.
\end{proof}

Now we prove Theorem \ref{thm:critetKum} by reducing to Theorem \ref{athm:Kum}.

\begin{proof}[Theorem~\ref{thm:critetKum}]
By Theorem~\ref{athm:Kum}, 
it suffices to show that,
for some finite unramified extension $K'$ of $K$, 
$X_{K'}$ can be written as $X_{K'} = \Km A$ for an abelian surface $A$ over $K'$.  

Let $A_{\overline K}$ be an abelian surface over $\overline K$ such that $X_{\overline K} = \Km (A_{\overline K})$, 
and let $Z_{\overline K}$ be the exceptional divisor of the blow-up $X_{\overline K} \rightarrow A_{\overline K}/\{\pm 1\}$.
By the description given at the beginning of this section, 
there is a double covering of $X_{\overline K}$ whose branch locus is $Z_{\overline K}$. 
By Lemma~\ref{lem:-2curves}, 
all the components of $Z_{\overline K}$ and the covering are defined over some finite unramified extension $K'$ of $K$.
Write $X_{\overline K} = X_{K'} \times_{K'} \overline K$ and $Z_{\overline K} = Z_{K'} \times_{K'} \overline K$.
Since we know that the inverse image of $Z_{K'}$ is, over $\overline K$,
the disjoint union of $16$ rational curves of self-intersection $-1$,
we can blow down this inverse image to get a variety $A$ over $K'$.
We see that $A \times_{K'} \overline K$ is the abelian surface $A_{\overline K}$ over $\overline K$ 
by the uniqueness of the double covering (Lemma~\ref{lem:2-cover}).
Moreover, 
the origin of $A_{\overline K}$, being the image of one of the contracted curves which was defined over $K'$, 
is a $K'$-rational point. 
Therefore we see that $A$ is an abelian surface over $K'$ such that $X_{K'} = \Km(A)$ over $K'$.
\end{proof}

\begin{remark}
 \label{rem:degboundKum}
We can give an explicit bound for the \emph{separable} degree of field extension.

By the proof of Theorem~\ref{thm:critetKum}, 
it suffices to estimate the degree of $K'$ such that $G_{K'}$ acts trivially on $\NS(X_{\overline K})$.
The rank of $\NS(X_{\overline K})$ is less than or equal to $22$ ($= \dim_{\bQ_l} H^2_\et(X_{\overline K}, \bQ_l)$).
Since every divisor on $X_{\overline K}$ can be defined over a finite extension of $K$, 
the image of $G_K$ in $GL(\NS(X_{\overline K}))$ is torsion. 
So it remains to give a bound for the order of a torsion subgroup of $GL(22, \bZ)$.

Take any prime number $l \geq 3$ (which we do not assume to be different from the characteristic).
By the exact sequence 
\[
 1
  \rightarrow 1 + l M(22, \bZ_l)
  \rightarrow GL(22, \bZ_l)
  \rightarrow GL(22, \bF_l)
  \rightarrow 1
\]
and the fact that $1 + l M(22, \bZ_l)$ is torsion-free, 
a torsion subgroup of $GL(22, \bZ)$
has order $\leq \lvert GL(22, \bF_l) \rvert \leq l^{22^2}$.
We can (by choosing $l=3$) take $3^{484}$ as a bound.
Of course this bound is too rough.
\end{remark}

\subsection*{Acknowledgements}
I am deeply grateful to my supervisor Atsushi Shiho for 
his invaluable support and also for helpful suggestions on section~\ref{sec:pproof}. 
I also thank Tetsushi Ito for the permission to use his unpublished master's thesis 
as an appendix. 
I would also like to thank Shouhei Ma and Takeshi Saito
for their helpful comments.

\end{document}